\documentclass[12pt,draft]{amsart}

\usepackage{amssymb,amscd}

\usepackage{enumerate}

\usepackage{graphicx}

\usepackage[all]{xy} 
\usepackage{mathtools} 

\usepackage{color} 

\makeatletter
\@namedef{subjclassname@2010}{%
  \textup{2010} Mathematics Subject Classification}
\makeatother

\newtheorem{thm}{Theorem}[section]
\newtheorem{cor}[thm]{Corollary}

\newtheorem{propo}[thm]{Proposition}

\newtheorem*{reduc2}{Reduction Theorem} 

\newtheorem*{funchom}{Homogeneous Function Theorem}

\newtheorem*{funthm}{First Fundamental Theorem of Sp}

\theoremstyle{definition}
\newtheorem{defin}[thm]{Definition}
\newtheorem{rem}[thm]{Remark}

\numberwithin{equation}{section}

\def\C{\mathcal{C}}

\def\F{\mathcal{F}}
\def\FF{\mathcal{F}'}

\def\RR{\mathbb{R}}
\def\R2n{\mathbb{R}^{2n}}

\def\Z2{\mathbb{Z}/2 \mathbb{Z}}

\def\Cs{\mathcal{C}^{\mathrm{sym}}}
\def\Cw{\C_{\eta}}

\def\T{\mathcal{T}}

\renewcommand{\d}{\ensuremath{\mathrm{d}}}
\def\dd{\ensuremath{\mathrm{d}}}

\def\TT{\mathbb{T}}

\def\qed{\hfill $\square$}

\def\Diff{\protect \mathrm{Diff}_{x_0}}

\def\Autw{\protect \mathrm{Aut}(\eta)}
\def\Auts0{\protect \mathrm{Aut}(s_0)}

\def\FDiff{\protect \mathrm{Aut}(\eta)_{x_0}}

\def\Gl{{\rm Gl}}

\def\Sl{{\rm Sl}}

\def\Sp{{\rm Sp}}
\def\Spp{{\rm Sp}(2n,\RR)}

\def\On{\mathrm{O}(s_+,s_-)}
\def\SOn{\mathrm{SO}(s_+,s_-)}

\newcommand{\Dif}{\protect \mbox{Diff}\,}
\newcommand{\Con}{\protect \mbox{Conn}\,}
\newcommand{\Fedm}{\protect \mbox{Fed}^m_{\scriptsize x_0}\,}
\newcommand{\Conw}{\protect \mbox{Conn}_{\eta}\,}

\newcommand{\Conwpinf}{\protect \mbox{Conn}_{\scriptsize \omega }\,}

\newcommand{\diffmmasdos}{\ensuremath{{\protect \mathrm{Diff}^{m+2}_{\scriptsize x_0}}}} 
\newcommand{\diffinfty}{\ensuremath{{\protect \mathrm{Diff}^{\infty}_{\scriptsize x_0}}}}

\newcommand{\ndiffmmasdos}{\ensuremath{{\protect \mathrm{NDiff}^{m+2}_{\scriptsize x_0}}}} 
\newcommand{\ndiffinfty}{\ensuremath{{\protect \mathrm{NDiff}^{\infty}_{\scriptsize x_0}}}}

\begin{document}

\baselineskip=17pt

\title[Invariant operations of a Fedosov structure]{On invariant operations of Fedosov structures}

\author[A. Gordillo-Merino]{{Adri\'{a}n}~Gordillo-Merino}
\address{Departamento de Didáctica de las Ciencias Experimentales y Matemáticas \\ Universidad de Extremadura \\ E-06071 Badajoz, Spain}
\email{adgormer@unex.es}\thanks{The authors have been partially supported by Junta de Extremadura and FEDER funds with project IB18087, as well as by projects GR18001 and GR21055 in the case of the second and third authors. The second author was additionally supported by the grant ``Plan 
Propio de Iniciación a la Investigación, Desarrollo Tecnológico e Innovación'' of Universidad de Extremadura.}

\author[R. Mart\'inez-Boh\'orquez]{Ra\'ul~Mart\'inez-Boh\'orquez}
\address{Departamento de Matem\'{a}ticas \\ Universidad de Extremadura \\ E-06071 Badajoz, Spain}
\email{raulmb@unex.es}

\author[J. Navarro]{Jos\'{e}~Navarro-Garmendia}
\address{Departamento de Matem\'{a}ticas \\ Universidad de Extremadura \\ E-06071 Badajoz, Spain}
\email{navarrogarmendia@unex.es}

\begin{abstract} In this paper we study invariant local operations that can performed on a Fedosov manifold, with a particular emphasis on tensor-valued operations (also known as natural tensors).
Our main result describes the spaces of homogeneous natural tensors as certain finite dimensional linear representations of the symplectic group. 

\end{abstract}



\date{\today}
\maketitle


\section{Introduction}

The notion of invariant operation has been key to the development of differential geometry and many of its applications. A paradigmatic example is its relevance in the early days of the nascent theory of General Relativity (\cite{M_BOOK}).  As time went by, the theory of these invariant operations evolved and produced significant mathematical results, such as the characterisation of the Pontryagin forms on Riemannian manifolds (\cite{GILKEY, ABP}) or the proof of the uniqueness of the Chern–Gauss–Bonnet formula (\cite{GILKEY_ANNALS}), both found by P. Gilkey during the mid-70s. 

In 1993, Kol\'a\v{r}-Michor-Slov\'ak (\cite{KMSBOOK}) published the monograph which has become the standard reference in this subject since then.  It summarises and enhances the main results and techniques that were known up to that point. However, this book is written with a functorial language that, outside specialists on the field,  has certainly not become standard; this has probably motivated that, in recent years, there have appeared various references that rewrite some of its most prominent results (\cite{FHBAMS, KMM, NS_DGA}).

Among the invariant operations that can be performed on a manifold,  tensor-valued operations are particularly relevant.  Also known as {\it natural tensors},  their description in the easiest possible terms has  always been a relevant question.  In presence of a linear connection, the main result of the theory describes these spaces of natural tensors as certain finite-dimensional linear representations of a classical Lie group: the linear groups $\Gl_n$ or $\Sl_n$ when considering natural tensors associated to linear connections (\cite{GMN_RACSAM}, \cite{GMN_MATH}, \cite{S_JDG}), the orthogonal groups $\On$ or $\SOn$ when considering natural tensors associated to pseudo-Riemannian metrics (\cite{NS_EINSTEIN}, \cite{T_JMP}), or the unitary groups $\mathrm{U}_n$ or $\mathrm{SU}_n$ for the corresponding  case of Kähler metrics (\cite{GPS_KAHLER}, \cite{T_JMP}). This description permits classical invariant theory to come into play and, in certain cases, to achieve this way an exhaustive computation of the spaces of natural tensors under consideration (see, for example, \cite{ABP},  \cite{GPS_RIEMANN}, \cite{GMN_RACSAM} or \cite{NS_EINSTEIN}).

Nevertheless, in this picture above,  the symplectic group was missing; in other words,  there was no theorem describing natural tensors associated to the so called Fedosov structures. 
Fedosov manifolds constitute the skew-symmetric version of Riemannian manifolds: they are defined as a triple $(X,\omega,\nabla)$, where $X$ is a smooth manifold of even dimension, $\omega$ is a symplectic form and $\nabla$ is a symplectic connection, that is, a symmetric linear connection such that $\nabla \omega = 0$. They are named after B. Fedosov,  who first constructed a canonical deformation quantization on these manifolds (\cite{FEDOSOV_1}, \cite{FEDOSOV}). 

The remedy to this situation started in 1998,  when Gelfand-Retakh-Shubin (\cite{GELFAND}) proved that any finite order, natural tensor associated to a Fedosov structure is indeed a function of the curvature and its successive derivatives.  This nice result, however, still had strong limitations:  it did not allow the use of the invariant theory of the symplectic group yet, and it imposed a strong finiteness hypothesis on the order of the local invariants.  

In this paper, we overcome this inconvenience and prove a statement (Theorem \ref{MainThmFed}) that describes natural tensors associated to Fedosov structures in terms of certain finite-dimensional linear representations of the symplectic group $\Sp (2n, \mathbb{R})$. Our theorem is completely analogous to the aforementioned results for linear connections or Riemannian metrics; in particular, it imposes no restrictions on the order of the natural tensors and it allows the use of classical invariant theory.  We plan to exploit these features in the future,  as it is plausible that they will allow the computation of interesting dimensional curvature identities,  analogous to those in \cite{GPS_RIEMANN} or \cite{GPS_KAHLER}, as well as another approach to moduli spaces of jets of Fedosov structures,  different to that used in  \cite{DUBROVSKIY}.

Finally, let us mention that the use of the language of sheaves and ringed spaces, much in the spirit of our previous works \cite{GMN_RACSAM} and \cite{GMN_MATH}, plays in this paper an essential role, especially to get rid of the finite order conditions of other developments.

\section{Statement of the Main Theorem}

Let $X$ be a smooth manifold of dimension $n$. Let $\Dif (X)$ denote the set of local diffeomorphisms\footnote{Throughout this text, the term \textit{diffeomorphism} will refer to a local diffeomorphism between two open subsets of a smooth manifold, unless explicitly otherwise stated.} between open subsets of $X$.

\begin{defin}
Let $\pi \colon F \rightarrow X$ be a (fibre) bundle over $X$. A  natural bundle over $X$ is a bundle $F \rightarrow X$ together with a map
\begin{align*}
\Dif (X) & \longrightarrow  \Dif (F)\\ 
\tau \ & \longmapsto  \ \tau_* \qquad ,
\end{align*}
called lifting of diffeomorphisms, satisfying the following properties \footnote{In the literature, a condition of regularity is added to the definition of natural bundle (the lifting of any smooth family of diffeomorphisms is smooth too). However, this property can be derived from the other two (see \cite{ET}). }:
\begin{itemize}
\item If $\tau \colon U \to V\,$ is a diffeomorphism between open subsets of $X$, then $\tau_* \colon F_U \to F_V\,$ is a diffeomorphism covering $\tau$, i.e. it makes the following square commutative:
\[
\xymatrix{
F_U \ \ar[r]^-{\tau_*}_{\sim} \ar[d]_-{\pi}  & \ F_V \ar[d]^-{\pi} \\
U \ar[r]^-{\tau}_{\sim} &  V   \ ,
} 
\] where $F_U := \pi^{-1} (U)$ and $F_V := \pi^{-1} (V)$.
\item Functoriality: $\mathrm{Id}_* = \mathrm{Id}$ and $(\tau \circ \tau')_*=( \tau)_*\circ (\tau')_*$.
\item Locality: for any diffeomorphism $\tau : U \rightarrow V$ and any open subset $U' \subset U$, $ (\tau_{|U'})_* = (\tau_*)_{|F_{U'}}$.
\end{itemize} 

%
\end{defin}


\begin{defin}
A natural sheaf $\F$ over $X$ is a subsheaf of the sheaf of smooth sections of a natural bundle $F\to X$ over $X$ such that, for any diffeomorphism $\tau: U \to V$, the morphism
\begin{align*}
\tau_* \colon  \F(U) &\longrightarrow \ \F(V)  \\
s \ & \longmapsto \ \tau_*\circ s \circ \tau^{-1}
\end{align*}
is well defined\footnote{Observe that we are committing an abuse of notation: we are denoting by $\tau_*$ both the lifting of $\tau$ to $F$ and the `action' of $\tau$ on $\F$. However, the context will help clarify which morphism we are working with. }.
\end{defin}

\noindent \textbf{Examples:} 
\begin{enumerate}
\item Let $F\to X$ be a natural bundle. It is easy to prove that the sheaf of smooth sections of $F$ is a natural sheaf, using that the lifting covers the lifted diffeomorphism. As such, the sheaf $\T_p^q$ of $(p,q)$-tensors over $X$ is a natural sheaf.
 
\item The Fedosov sheaf, defined on any open subset $U\subseteq X$ as
$$\F(U):=\{(\omega,\nabla)\in (\Lambda^2 \times \Cs) (U): \nabla \omega = 0\} \ ,$$
is a natural sheaf, where $\Lambda^2$ denotes the sheaf of non-singular 2-forms on $X$ and $\Cs$ denotes the sheaf of symmetric linear connections on $X$. Observe that the condition $\nabla \omega = 0$ is natural: if $(\omega, \nabla) \in \F (U)$, then $(\tau_* \nabla)(\tau_* \omega)=0$.   
\end{enumerate}

\begin{defin}
Let $\F$ and $\FF$ be natural sheaves over $X$. A morphism of sheaves $\phi : \F \rightarrow \FF$ is natural if it is regular\footnote{The regularity condition is technical in nature, and as such it will be properly defined in Section \ref{SectionPeetre}, Definition \ref{DefiRegular}.} and commutes with the action of diffeomorphisms on sections; that is to say, if for any diffeomorphism $\tau :U \rightarrow V$, the following square commutes:
\begin{equation}
\xymatrix{
\F(U) \ar[r]^\phi \ar[d]_{\tau_*}   & \FF (U) \ar[d]^{\tau_*} \\
\F (V) \ar[r]^\phi & \FF (V) \, ,
}
\end{equation}
where $\tau_* \colon \F (U) \to \F (V)\,$ is defined as follows:
\begin{align*}
\tau_* \colon  \F(U) &\longrightarrow \ \F(V)  \\
s \ & \longmapsto \ \tau_*\circ s \circ \tau^{-1} \ .
\end{align*}
\end{defin}

\begin{defin}
A natural morphism of sheaves $\F \rightarrow \T$ between the Fedosov sheaf $\F$ and a sheaf of tensors $\T$ over $X$ is called a natural tensor (associated to Fedosov structures). 
\end{defin}

%

A condition of homogeneity is required to guarantee that the natural tensors depend on a finite amount of variables only:

\begin{defin}
Let $\delta\in \RR$. We say that a natural tensor $T:\F \rightarrow \T$ is homogeneous of weight $\delta$ if, for all non-zero $\lambda\in \RR$, it holds that\footnote{Observe that if $(\omega,\nabla)$ is a Fedosov structure, then $(\lambda \omega , \nabla)$ is also a Fedosov structure for any $\lambda \in \RR \setminus
 \{0\}$.}:
$$T(\lambda^2 \omega, \nabla)=\lambda^\delta T(\omega, \nabla) \ .
$$
\end{defin}

\noindent \textbf{Examples:} 
\begin{itemize}
\item The symplectic form can be understood as a natural $(2,0)$-tensor associated to Fedosov structures whose value on a Fedosov structure $(\omega, \nabla)$ is $\omega$. It is homogeneous of weight $2$.
\item The $(4,0)$ curvature operator, defined as a natural $(4,0)$-tensor whose value on a Fedosov structure $(\omega,\nabla)$ defined on an open set $U\subset X$ is:
$$ R_{(\omega,\nabla)} (D_1 , D_2 , D_3 , D_4) := \omega(D_1, \nabla_{D_3} \nabla_{D_4} D_2 -\nabla_{D_4} \nabla_{D_3} D_2 - \nabla_{[D_3 , D_4]} D_2 ) \ , $$ 
 which is an homogeneous tensor of weight $2$.
\end{itemize}

The following result, whose proof will be detailed during Section \ref{SectionMain}, describes all natural tensors associated to Fedosov structures:

\begin{thm} \label{MainThmFed} 
Let $X$ be a smooth manifold of dimension $2n$, and let $\F$ denote the sheaf of Fedosov structures. Let $\T$ be the sheaf of $p$-covariant tensors over $X$. Let $\delta \in \mathbb{Z}$.


Fixing a point $x_0 \in X$ and a chart $U\simeq \R2n$ around $x_0$ produces a $\mathbb{R}$-linear isomorphism

$$
\begin{CD}
\left\{
\begin{array}{c}
 \text{Natural morphisms of sheaves} \ \\
 \F \longrightarrow \T \ \\
 \text{homogeneous of weight }\delta \
\end{array} \right\} @=
\bigoplus \limits_{d_1, \ldots , d_r} \mathrm{Hom}_{\Sp}(S^{d_1}N_1 \otimes \ldots \otimes S^{d_r}N_r , T_{x_0} ) \ ,
\end{CD}
$$
where $\Sp=\Spp$ denotes the symplectic group, $T_{x_0}$ denotes the vector space of $p$-covariant tensors at $x_0$ and $d_1, \ldots , d_r$ run over the non-negative integer solutions of the equation
\[
2d_1 + \ldots + (r+1)d_r =p-\delta \ .
\] 
\end{thm}

The spaces $N_m$ are called spaces of normal tensors of symplectic connections, and they are vector spaces made of tensors which recover the symmetries of the functions $\Gamma_{ijk}:=\omega_{il} \Gamma^l_{jk}$\footnote{During this work, we will follow Einstein summation convention, unless the summation is explicitly stated.}, where $\Gamma^l_{jk}$ are the Christoffel symbols of a symplectic connection in normal coordinates at the point $x_0$. They will be rigorously defined during Section \ref{SectionInvariants}.

\bigskip

\section{The Peetre-Slov\'ak Theorem}\label{SectionPeetre}

Let us briefly introduce the category of ringed spaces: they generalise smooth manifolds in a way that allows us to consider infinite dimensional spaces or quotients of smooth manifolds by the actions of groups.

\begin{defin}
A  ringed space is a pair $(X, \mathcal{O}_X)$, where $X$ is a topological space and $\mathcal{O}_X$ is a sub-algebra of the sheaf of real-valued continuous functions on $X$.

A morphism of ringed spaces\footnote{By similarity with the category of smooth manifolds, we will often call morphisms of ringed spaces as {\it smooth morphisms}.} $\varphi \colon (X , \mathcal{O}_X) \to (Y , \mathcal{O}_Y)$ is a continuous map $\varphi \colon X \to Y $ such that composition with $\varphi$ induces a morphism of sheaves $ \varphi^* \colon \mathcal{O}_Y \to \varphi_* \mathcal{O}_X$, that is, for any open set $V \subset Y$ and any function $f \in \mathcal{O}_Y (V)$, the composition $f \circ \varphi $ lies in $\mathcal{O}_X ( \varphi^{-1} V)$.
\end{defin}

The two main properties of this category that we will make use of are the  existence of inverse limits and the existence of quotients by the action of a group. For example, if $F\rightarrow X$ is a fibre bundle over a smooth manifold $X$, then the space $J^\infty F$ of $\infty$-jets of sections of $F\rightarrow X$ is defined as the inverse limit of the sequence of $k$-jets fibre bundles:
$$ \ldots \to J^k F \to J^{k-1}F \to \ldots \to F \to X \ . $$
The spaces $J^k F$ are smooth manifolds, and thus they are ringed spaces, choosing as sheaf the sheaf of real-valued smooth functions. Therefore, the space $J^\infty F$ is canonically imbued with a structure of ringed space. This fact will become of great relevance in the Peetre-Slov\'ak theorem, where natural tensors will be related to morphisms of ringed spaces coming from an $\infty$-jet space.

Additionally, we will require the following corollary:
 
\begin{cor}\label{CorolarioCociente}
Let $G\,$ be a group acting on two ringed spaces $X$ and $Y$, and let $H\subseteq G$ be a subgroup that acts trivially on $Y$.

Then, the universal property of the quotient restricts to a bijection:
$$
\begin{CD}
\left\{
\begin{aligned}
& \hskip .1cm G \text{-equivariant morphisms }  \ \\
&  \text{ of ringed spaces  } X \to Y \, \ 
\end{aligned} \right\} @=
\left\{
\begin{aligned}
&\hskip .1cm G/H \text{-equivariant morphisms  }  \ \\
& \text{ of ringed spaces } X/H \longrightarrow Y \ 
\end{aligned} \right\} \ .
\end{CD}
$$
\end{cor}

\smallskip

Now,  let us define a sort of ``smoothness'' condition for morphisms of sheaves:

\begin{defin}\label{DefiRegular}
Let $\F$ and $\FF$ be (sub)sheaves of the sheaves of smooth sections of the fibre bundles $F\rightarrow X$ and $F' \rightarrow X$, and let $T$ be a smooth manifold. A morphism of sheaves $\phi: \F \rightarrow \FF$ is said to be regular if,  for any smooth family of sections $\{s_t:U \rightarrow F\}_{t\in T}$ such that $U\simeq \RR^n$ and $s_t \in \F(U)$ for all $t\in T$,  the family $\{\phi(s_t):U \rightarrow F'\}_{t\in T}$ is also smooth.
\end{defin}

The Peetre-Slov\'ak Theorem (\cite{KMSBOOK, NS_PEETRE}) assures that any natural morphism of sheaves is a natural differential operator:


\begin{thm}[Peetre-Slov\'{a}k]\label{PeetreSlovak}
Let $X$ be a smooth manifold. Let $F'\rightarrow X$ and $F''\rightarrow X$ be natural bundles over $X$, and let $\F'$ and $\FF'$ be their respective sheaves of smooth sections over $X$. 

The choice of a point $\,p\in X\,$ allows to define this bijection:
$$
\begin{CD}
\left\{
\begin{aligned}
& \, \text{Natural morphisms of sheaves} \ \\
&\hskip 1.7cm  \phi \colon \F' \longrightarrow \mathcal{F}'' \ 
\end{aligned} \right\} @=
\left\{
\begin{aligned}
&\hskip .1cm  \Diff \text{-equivariant smooth maps }  \ \\
&\hskip 1.8cm  J^\infty_p F' \longrightarrow F''_p \ 
\end{aligned} \right\} \ ,
\end{CD}
$$

\noindent where $\Diff$ stands for the group of germs of diffeomorphisms $\tau$ between open sets of $X$ such that $\tau (p) = p$.
\end{thm}

\section{Natural Operations on a Fedosov Structure}

Let $X$ be a smooth manifold of dimension $2n$. Let $\F_X$ and $\T_X$ be the sheaves of Fedosov structures and $p$-covariant tensors over $X$, respectively.




\begin{propo}\label{red1}
The choice of a chart $U\subseteq X$ gives a bijection:
$$
\begin{CD}
\left\{
\begin{aligned}
&\hskip .1cm   \text{Natural morphisms of sheaves}  \ \\
&\hskip 1.8cm   \F_X \longrightarrow \T_X \ 
\end{aligned} \right\} @=
\left\{
\begin{aligned}
&\hskip .1cm  \text{Natural morphisms of sheaves}  \ \\
&\hskip 1.8cm  \F_{\R2n}  \longrightarrow \T_{\R2n} \ 
\end{aligned} \right\} \ .
\end{CD}
$$ 
\end{propo}

\begin{proof}
Let $U\subseteq X$ be any chart, so that $U\simeq \R2n$. We will prove that there exists a bijection:

$$
\begin{CD}
\left\{
\begin{aligned}
&\hskip .1cm   \text{Natural morphisms of sheaves}  \ \\
&\hskip 1.8cm  \phi: \F_X \longrightarrow \T_X \ 
\end{aligned} \right\} @=
\left\{
\begin{aligned}
&\hskip .1cm  \text{Natural morphisms of sheaves}  \ \\
&\hskip 1.8cm \phi_U : \F_{U}  \longrightarrow \T_{U} \ 
\end{aligned} \right\} \ ,
\end{CD}
$$ 

thus obtaining the desired result.

For any natural morphism of sheaves $f: \F_{U}  \rightarrow \T_{U}$, let us construct the corresponding natural morphism of sheaves $\phi_f: \F_X \rightarrow \T_X$: for any $s\in \F_X(V)$ and $x\in V$, we must define $\phi_f(s)(x)$. 

As $\phi_f(s)(x)=\phi_f(s_{|W})(x)$ for any $W\subseteq V$ containing $x$, we may suppose that $V$ is also a chart, thus obtaining a local isomorphism $\tau:V\rightarrow U$, and so we may define:

$$\phi_f(s)(x)=\tau_*^{-1}(f(\tau_*s))(x)$$

It is trivial to check that this morphism is well defined, natural, regular and the inverse of the map $\phi \rightarrow \phi_{|U}$. \qed

\end{proof}

\bigskip

Let $(x_1, y_1, \ldots, x_n, y_n)$ be global coordinates on $\R2n$, and set $\eta=\d x_1 \wedge \d y_1 + \ldots + \d x_n \wedge \d y_n$. Let $\Con_{\eta} \rightarrow \R2n$ be the fibre bundle of symplectic connections for the symplectic form $\eta$, which is an affine subbundle of $\Con \rightarrow \R2n$. Let $\Cw$ be the sheaf of smooth sections of $\Conw$.

\begin{propo}\label{red2}
With the previous notations, there exists a bijection:
$$
\begin{CD}
\left\{
\begin{aligned}
&\hskip .1cm   \text{Natural morphisms of sheaves}  \ \\
&\hskip 1.8cm  \F_{\R2n}  \longrightarrow \T_{\R2n} \ 
\end{aligned} \right\} @=
\left\{
\begin{aligned}
&\hskip .1cm   \Autw \text{-natural morphisms of sheaves}  \ \\
&\hskip 1.8cm  \Cw  \longrightarrow \T_{\R2n} \ 
\end{aligned} \right\} \ ,
\end{CD}
$$ 
where a natural morphism of sheaves $\phi: \Cw  \rightarrow \T$ is said to be $\Autw$-natural if it is regular and verifies the naturalness condition for any local diffeomorphism $\tau: U \rightarrow V$ between open sets of $\R2n$ such that $\tau \cdot (\eta_{|_U}) = \eta_{|_V}$.
\end{propo}

\begin{proof}
Given a natural morphism of sheaves $\phi: \F \rightarrow \T$, the corresponding morphism of sheaves $\hat{\phi}: \Cw  \rightarrow \T$ is given, at any open subset $U\subseteq \R2n$, by
$$\hat{\phi}_U (\nabla) := \phi_U (\eta, \nabla) \ ,$$
which is trivially an $\Autw$-natural morphism of sheaves.

Let us give the inverse map, that is, to define a natural morphism of sheaves $\tilde{\varphi}: \F \rightarrow \T$ from  an $\Autw$-natural morphism of sheaves $\varphi:\Cw \rightarrow \T_{\R2n}$. Let $(\omega, \nabla) \in \F_{\R2n}(U)$ and $x\in U$. There exists an open subset $V\subseteq U$ and a diffeomorphism $\tau: V \rightarrow V$ such that $x\in V$ and $\tau \cdot (\eta_{|_V})=\omega_{|_U}$. As the value at $x$ of $\tilde{\varphi}(\omega, \nabla)$ does not depend on the neighbourhood of $x$ chosen, we may assume that $V=U$. Then:
$$\tilde{\varphi}(\omega,\nabla)(x):=\tau \cdot \varphi (\tau^{-1} \cdot \nabla) (x) \ .$$ \qed
\end{proof}

\begin{cor}\label{red3}
The choice of a point $x_0\in \R2n$ produces a bijection:
$$
\begin{CD}
\left\{
\begin{aligned}
&\hskip .1cm   \Autw \text{-natural morphisms of sheaves}  \ \\
&\hskip 1.8cm  \Cw  \longrightarrow \T_{\R2n} \ 
\end{aligned} \right\} @=
\left\{
\begin{aligned}
&\hskip .1cm   \Autw_{x_0} \text{-equivariant smooth maps}  \ \\
&\hskip 1.8cm  J_{x_0}^\infty \Conw  \longrightarrow T_{x_0} \ 
\end{aligned} \right\} ,
\end{CD}
$$ 
where $\Autw_{x_0}$ denotes the group of germs of diffeomorphisms $\tau$ between open sets of $\R2n$ such that $\tau (p)=p$ and $\tau \cdot \eta = \eta$.
\end{cor}

\begin{proof}
A simple variation of the Peetre-Slov\'ak theorem \ref{PeetreSlovak}, substituting naturalness by $\Autw$-naturalness, allows us to conclude. 
\qed
\end{proof}

However, even though fixing a symplectic form in a neighbourhood of a point allows us to use the Peetre-Slov\'ak Theorem -- reducing the computations to the $\infty$-jet space, the resulting space is difficult to reduce. It is convenient to take a step back, unfixing the symplectic form, in order to advance:

\begin{propo}
There exists a bijection:
$$
\begin{CD}
\left\{
\begin{array}{c}
   \Autw_{x_0} \text{-equivariant smooth maps}  \ \\
   \\
  J_{x_0}^\infty \Conw  \longrightarrow T_{x_0}
\end{array} \right\} @=
\left\{
\begin{array}{c}
   \Diff \text{-equivariant smooth maps}  \ \\
   \\
  J_{x_0}^\infty \F \longrightarrow T_{x_0}
\end{array} \right\} \ ,
\end{CD}
$$ 
where $J_{x_0}^\infty \F := \{(j_{x_0}^{\infty} \omega, j_{x_0}^\infty \nabla): (\omega,\nabla)\in \F_{x_0}\}$.
\end{propo}
\begin{proof}
The proof of this result is similar to that of Proposition \ref{red2}.\qed
\end{proof}
Later on, only the value of the symplectic form at $x_0$ will be fixed, as the rest of the $\infty$-jet will be determined by the compatibility condition with the $\infty$-jet of a symplectic connection.

\begin{rem}
Observe that $J^\infty_{x_0} \F$ coincides with the set $$\{(j_{x_0}^{\infty} \omega, j_{x_0}^\infty \nabla)\in J_{x_0}^{\infty} \Lambda^2 \times J_{x_0}^\infty \Con^{\mathrm{sym}}: j_{x_0}^\infty (\nabla \omega)=0 \} \ .$$ The reasoning goes as follows: due to the formal version of the Poincaré Lemma, the $\infty$-jet of a non-singular 2-form $\omega$ such that $j_{x_0}^\infty (\nabla \omega)=0$ verifies that $j_{x_0}^\infty \omega = j_{x_0}^\infty (\dd \theta)$, for some 1-form $\theta$ defined on a neighbourhood of $x_0$. Therefore, $j_{x_0}^\infty \omega$ can be extended to a symplectic form at a neighbourhood of $x_0$ (considering, for example, $\dd \theta$). Then, a symplectic connection extending $j_{x_0}^\infty \nabla$ can be chosen, as symplectic connections compatible with a fixed symplectic form constitute a fibre bundle.
\end{rem}

\bigskip
\section{Invariants of Symplectic Connections}\label{SectionInvariants}

Let $x_0\in X$, let $( \omega, \nabla)$ be the germ of a Fedosov structure at $x_0$, and let $ \bar{\nabla}\,$ be the germ of the flat connection at $x_0 \in X$ corresponding, via the exponential map, to the flat connection of $T_{x_0} X$. Let $\TT:= C_2^1 (\omega \otimes \TT)$, where $C_i^j$ denotes the tensor contraction of the $i$-th covariant index with the $j$-th contravariant index.

\begin{defin}
For any integer $m\geq 0$, the $m$-th normal tensor of $\nabla$ at $x_0$ is $\bar{\nabla}_{x_0}^m \TT$.
\end{defin}

In a system of normal coordinates $\,(x_1,\ldots,x_n)\,$ around the point $x_0$ for $\nabla$, the tensor $\,\bar{\nabla}_{x_0}^m \TT$ is written as
\[
\bar{\nabla}_{x_0}^m \TT= \sum_{i,j,k,a_1,\ldots,a_m}  \Gamma_{ijk,a_1\ldots a_m} \cdot \mathrm{d}_{x_0} x_i \otimes  \mathrm{d}_{x_0} x_j \otimes  \mathrm{d}_{x_0} x_k \otimes  \mathrm{d}_{x_0} x_{a_1} \otimes \ldots \otimes  \mathrm{d}_{x_0} x_{a_m} \ ,
\]
where $\Gamma_{ijk,a_1\ldots a_m}^k:=\frac{\partial^m \Gamma_{ijk}}{\partial x_{a_1} \ldots \partial x_{a_m}} (x_0)$ and $\Gamma_{ijk}=\sum_{l=1}^{2n} \omega_{il} \Gamma_{jk}^l$.

\begin{rem}\label{LowerIndices}
 Notice that the sequences $\bar{\nabla}_{x_0}^1 \TT,\ldots, \bar{\nabla}_{x_0}^m \TT$ and $\bar{\nabla}_{x_0}^1 (\nabla-\bar{\nabla}),\ldots, \break \bar{\nabla}_{x_0}^m (\nabla-\bar{\nabla})$ mutually determine each other, as $\omega$ is non-singular. Following the notations above, the tensor $\bar{\nabla}_{x_0}^m (\nabla-\bar{\nabla})$ is written as usual:
\[
\bar{\nabla}_{x_0}^m \TT= \sum_{i,j,k,a_1,\ldots,a_m}  \Gamma_{ij,a_1\ldots a_m}^k \cdot \left( \frac{\partial}{\partial x_k} \right)_{x_0} \otimes  \mathrm{d}_{x_0} x_i \otimes  \mathrm{d}_{x_0} x_j \otimes  \mathrm{d}_{x_0} x_{a_1} \otimes \ldots \otimes  \mathrm{d}_{x_0} x_{a_m} \ ,
\]
where $\Gamma_{ij,a_1\ldots a_m}^k:=\frac{\partial^m \Gamma_{ij}^k}{\partial x_{a_1} \ldots \partial x_{a_m}} (x_0)$. 
\end{rem}


\begin{defin} 
The  space $N_{m}\,$ of normal tensors of order $m$ at $x_0\in X$ is the vector subspace of $(m+3)$-tensors whose elements $T$ verify the following symmetries:

\begin{enumerate}
\item they are symmetric in the second and third indices, and in  the last $m$:
$$T_{ikja_1\ldots a_m}=T_{ijka_1\ldots a_m}, \quad T_{ijka_{\sigma(1)}\ldots a_{\sigma(m)}}=T_{ijka_1\ldots a_m}, \quad \forall \sigma \in S_m \ ;$$
\item the symmetrization of the last $m+2$ covariant indices is zero:
$$\sum_{\sigma \in S_{m+2}} T_{i\sigma(j)\sigma(k)\sigma(a_1)\ldots\sigma(a_m)=0} \ ;$$
\item the following tensor is symmetric in $k$ and $a_1$:
$$T_{ikja_1\ldots a_m}-T_{jkia_1\ldots a_m} \ .$$
\end{enumerate}
\end{defin} 

Due to its symmetries, it is immediate that $N_0=0$.

Normal tensors belong in $N_m$, that is, $ \,\bar{\nabla}_{x_0}^m  \TT \, \in N_{ m}\,$, due to its expression in normal coordinates  (\cite{GELFAND}). As the tensor $\bar{\nabla}_{x_0}^m \TT$ depends only on the value of the $m$-jet $j_{x_0}^m \nabla$, the following map is well-defined:

\begin{align*}
\phi_m \colon J_{x_0}^m \F &\longrightarrow \ \ \Lambda_0 \times \prod\limits_{i=1}^m N_i  \\
(j^{r+1}_{x_0} \omega,j^r_{x_0} \nabla) \ & \xmapsto{\hspace*{0.6cm}} \ (\omega_{x_0}, \bar{\nabla}_{x_0}^1 \TT, \bar{\nabla}_{x_0}^2 \TT, \ldots,  \bar{\nabla}_{x_0}^m \TT\,) \, ,
\end{align*}

where $\Lambda_0$ denotes the open set of non-singular 2-forms at $x_0$. 

The maps $\phi_m$ are $\Diff$-equivariant and compatible, meaning that they commute with the restrictions $ J_{x_0}^m \F \rightarrow  J_{x_0}^{m-1} \F$ and $\Lambda_0 \times \prod\limits_{i=1}^m N_i \rightarrow \Lambda_0 \times \prod\limits_{i=1}^{m-1} N_i$. Therefore there exists a morphism of ringed spaces:

\begin{align*}
\phi_\infty \colon J_{x_0}^\infty \F &\longrightarrow \ \ \Lambda_0 \times \prod\limits_{i=1}^\infty N_i  \\
(j^\infty_{x_0} \omega, j^\infty_{x_0} \nabla) \ & \xmapsto{\hspace*{0.6cm}} \ (\omega_{x_0}, \bar{\nabla}_{x_0}^1 \TT, \bar{\nabla}_{x_0}^2 \TT, \ldots \,) \, .
\end{align*}
\medskip

\begin{reduc2} \label{ReductionTheorem2}
The $\diffmmasdos -$equivariant morphism of ringed spaces
\begin{align*}
\phi_m \colon J_{x_0}^m \F &\longrightarrow \ \ \Lambda_0 \times \prod\limits_{i=1}^m N_i  \\
(j^{m+1}_{x_0} \omega,j^m_{x_0} \nabla) \ & \xmapsto{\hspace*{0.6cm}} \ (\omega_{x_0}, \bar{\nabla}_{x_0}^1 \TT, \bar{\nabla}_{x_0}^2 \TT, \ldots,  \bar{\nabla}_{x_0}^m \TT\,) \, .
\end{align*}
is surjective, its fibres are the orbits of $\ndiffmmasdos$ and it admits smooth sections passing through any point of $J_{x_0}^m \F$.

As a consequence, $\phi_m$ induces a $\Gl$-equivariant isomorphism of ringed spaces:
$$ (J_{x_0}^m \F \ )/\ndiffmmasdos =\joinrel= N_1 \times \ldots \times N_m\, \, .$$
\end{reduc2}

\begin{proof}
Let us first prove that the fibres of $\phi_m$ are the orbits of $\ndiffmmasdos$. Let $\, (j^{m+1}_{x_0} \omega, j_{x_0}^m \nabla) \,$, $\,(j^{m+1}_{x_0} \omega', j_{x_0}^m \nabla')\,$ be two points in the orbit of $\ndiffmmasdos$, that is, $\,(j^{m+1}_{x_0} \omega', j_{x_0}^m \nabla')= j^{m+2}_{x_0} \tau \cdot (j^{m+1}_{x_0} \omega, j_{x_0}^m \nabla)\,$ for some $j^{m+2}_{x_0} \tau \in \ndiffmmasdos$.  As $\ndiffmmasdos$ acts by the identity on $\Lambda_0 \times \prod\limits_{i=1}^m N_i$,

$$\phi_m (j^{m+2}_{x_0} \tau \cdot (j^{m+1}_{x_0} \omega, j_{x_0}^m \nabla)\,) = j^{m+2}_{x_0} \tau \cdot \phi_m ((j^{m+1}_{x_0} \omega, j_{x_0}^m \nabla)\,) = \phi_m ((j^{m+1}_{x_0} \omega, j_{x_0}^m \nabla)\,) \ .$$

Let now $\, (j^{m+1}_{x_0} \omega, j_{x_0}^m \nabla) \,$, $\,(j^{m+1}_{x_0} \omega', j_{x_0}^m \nabla')\, \in J_{x_0}^m \F$ be two points in the same fibre of  $\phi_m$, that is, $\phi_m ((j^{m+1}_{x_0} \omega, j_{x_0}^m \nabla))=\phi_m ((j^{m+1}_{x_0} \omega', j_{x_0}^m \nabla'))= (T_1,\ldots,T_r) $. Let us fix a base of $T_{x_0} X$, let $x_1, \ldots, x_{2n}$ and $x_1',\ldots,x_{2n}'$ be the systems of normal coordinates induced by the fixed base for $j_{x_0}^m \nabla$ and $j_{x_0}^m \nabla'$, respectively, and let $\tau$ be the diffeomorphism that verifies $\tau \cdot x_i=x_i'$ for all $i\in \{1,\ldots,2n\}$. As $\dd_{x_0} x_i = \dd_{x_0} x_i'$ for all $i\in \{1,\ldots,2n\}$, it holds that $j_{x_0}^{m+2} \tau \in \ndiffmmasdos$. 

Let us write 
$$j^m_{x_0} \nabla = (\,0\,, \Gamma_{ij,a_1}^k, \ldots, \Gamma_{ij,a_1\ldots a_m}^k), \quad j^{m+1}_{x_0} \omega = (\omega_{ij}, \omega_{ij,k}, \ldots, \omega_{ij,ka_1\ldots a_m})$$
in the coordinates induced by $x_1, \ldots, x_{2n}$ on $J_{x_0}^m \F$. Similarly, in the coordinates induced by $x_1', \ldots, x_{2n}'$ on $J_{x_0}^m \F$, we write 
$$j^m_{x_0} \nabla' = (\,0\, , (\Gamma')_{ij,a_1}^k, \ldots, (\Gamma')_{ij,a_1\ldots a_m}^k),$$ 
$$j^{m+1}_{x_0} \omega' = \break (\omega'_{ij}, \omega'_{ij,k}, \ldots, \omega'_{ij,ka_1\ldots a_m}),$$  $$j_{x_0}^{m}(\tau \cdot \nabla) = (\,0\,, (\tau \cdot \Gamma)_{ij,a_1}^k, \ldots, (\tau \cdot \Gamma)_{ij,a_1\ldots a_m}^k),$$ $$j^{m+1}_{x_0} (\tau \cdot \omega) = ((\tau \cdot \omega)_{ij}, (\tau \cdot \omega)_{ij,k}, \ldots, (\tau \cdot \omega)_{ij,ka_1\ldots a_m}).$$

For all $r\in \{1,\ldots,m\}$, using that $j_{x_0}^{m+2}\tau \in \ndiffmmasdos$ we obtain the following equalities:
\begin{align*}
&\sum_{i,j,k,a_1,\ldots,a_r}  (\Gamma')_{ijk,a_1\ldots a_r}  \mathrm{d}_{x_0} x_i' \otimes  \mathrm{d}_{x_0} x_j' \otimes  \mathrm{d}_{x_0} x_k' \otimes  \mathrm{d}_{x_0} x_{a_1}' \otimes \ldots \otimes  \mathrm{d}_{x_0} x_{a_r}' =T_r  \\
 &= \tau \cdot T_r = \sum_{i,j,k,a_1,\ldots,a_r}  (\tau \cdot \Gamma)_{ijk,a_1\ldots a_r}  \mathrm{d}_{x_0} x_i' \otimes  \mathrm{d}_{x_0} x_j' \otimes \mathrm{d}_{x_0} x_k' \otimes  \mathrm{d}_{x_0} x_{a_1}' \otimes \ldots \otimes  \mathrm{d}_{x_0} x_{a_r}' \ ,
\end{align*}
and so $(\tau \cdot \Gamma)_{ijk,a_1\ldots a_r}= (\Gamma')_{ijk,a_1\ldots a_r}$ for all $r\in \{1,\ldots,m\}$. 

Thus, by Remark \ref{LowerIndices}, it is now enough to check that $j_{x_0}^m(\tau \cdot \omega )=j_{x_0}^m(\omega')$:

$$\sum_{i<j} (\tau \cdot \omega)_{ij} \dd_{x_0} x_i' \wedge \dd_{x_0} x_j' = (\tau \cdot \omega )_{x_0} =  \omega_{x_0} = \omega'_{x_0} = \sum_{i<j} \omega_{ij} \dd_{x_0} x_i' \wedge \dd_{x_0} x_j'   \ ,$$
$$
(\tau \cdot \omega)_{ij,k}= (\tau \cdot \Gamma)_{ikj} - (\tau \cdot \Gamma)_{jki} = 0 =  (\Gamma')_{ikj} - (\Gamma')_{jki} = \omega'_{ij,k} \ ,
$$
$$\vdots$$
$$
(\tau \cdot \omega)_{ij,ka_1\ldots a_m}= (\tau \cdot \Gamma)_{ikj,a_1\ldots a_m} - (\tau \cdot \Gamma)_{jki,a_1\ldots a_m} =  (\Gamma')_{ikj,a_1\ldots a_m} - (\Gamma')_{jki,a_1\ldots a_m} = \omega'_{ij,ka_1\ldots a_m} \ .
$$

Lastly, let us prove the statement about the existence of smooth sections 
$$s:\Lambda_0 \times N_1 \times \ldots \times N_m\,.$$
Let us fix a system of coordinates $x_1,\ldots,x_{2n}$ at $x_0$, and let \break $(B_{ij},A_{ijka_1}, \ldots, A_{ijka_1\ldots a_m}) \in \Lambda_0 \times N_1 \times \ldots \times N_m$. 

The jet $s((B_{ij},A_{ijka_1}, \ldots, A_{ijka_1\ldots a_m}))=(j_{x_0}^{m+1} \omega, j_{x_0}^m \nabla)$ is defined, in the coordinates induced by the fixed system in $J_{x_0}^m \F$, as follows:
$$\Gamma_{ijk}=0, \Gamma_{ijk,a_1}=A_{ijka_1}, \ldots, \Gamma_{ijk,a_1\ldots a_m}=A_{ijka_1\ldots a_m} \ ,$$
$$\omega_{ij}=B_{ij}, \omega_{ij,k}=0, \omega_{ij,ka_1}=A_{ikja_1} - A_{jkia_1}, \ldots, \omega_{ij,ka_1\ldots a_m}=A_{ikja_1\ldots a_m} -A_{jkia_1\ldots a_m} \ , $$
and so the jet $j_{x_0}^m \nabla= (\,0\,, \Gamma_{ij,a_1}^k, \ldots, \Gamma_{ij,a_1\ldots a_m}^k)$ is defined. The symmetries of the spaces $N_i$ assure that $(j_{x_0}^{m+1} \omega, j_{x_0}^m \nabla)\in \Fedm = J_{x_0}^m \F$ and that $x_1,\ldots,x_{2n}$ is a system of normal coordinates at $x_0$ for $j_{x_0}^m \nabla$. 

\qed
\end{proof}

\begin{cor}\label{CoroRedThm} 
The $\diffinfty -$equivariant morphism of ringed spaces
\begin{align*}
\phi_\infty \colon J_{x_0}^\infty \F  &\longrightarrow \ \ \Lambda_0 \times \prod\limits_{i=1}^\infty N_i\,  \\
(j^\infty_{x_0} \omega, j^\infty_{x_0} \nabla) \ & \xmapsto{\hspace*{0.6cm}} \ ( \omega_{x_0}, \bar{\nabla}_{x_0}^1 \TT, \bar{\nabla}_{x_0}^2 \TT, \ldots \,) ,
\end{align*}
induces a $\Gl$-equivariant isomorphism of ringed spaces:
$$ (J_{x_0}^\infty \F )/\ndiffinfty =\joinrel=\Lambda_0 \times \prod\limits_{i=1}^\infty N_i\, \, .$$
\end{cor}

\begin{cor}
The choice of a non-singular 2-form $\eta_{x_0}$ at $x_0$ produces a bijection:
$$
\begin{CD}
\left\{
\begin{array}{c}
 \Gl \text{-equivariant smooth maps}  \ \\
 \Lambda_0 \times \prod\limits_{i=1}^\infty N_i\,  \longrightarrow T_{x_0} \  
\end{array} \right\} @=
\left\{
\begin{array}{c}
  \Spp \text{-equivariant smooth maps}  \ \\
  \prod\limits_{i=1}^\infty N_i  \longrightarrow T_{x_0}  \ 
\end{array} \right\} \ ,
\end{CD}
$$
where $\Spp :=  \{  \mathrm{d}_{x_0} \tau \colon \tau \in \FDiff \}$.
\end{cor}
\begin{proof}
The proof of this result is similar to Proposition \ref{red2}.
\end{proof}

\bigskip


\section{Proof of Theorem \ref{MainThmFed}}\label{SectionMain}

\begin{defin}
Let $\delta\in \RR$. We say that a natural tensor $T:\F \rightarrow \T$ is homogeneous of weight $\delta$ if, for all non-zero $\lambda\in \RR$, it holds that\footnote{Observe that if $(\omega,\nabla)$ is a Fedosov structure, then $(\lambda \omega , \nabla)$ is also a Fedosov structure for any $\lambda \in \RR \setminus
 \{0\}$.}:
$$T(\lambda^2 \omega, \nabla)=\lambda^\delta T(\omega, \nabla) \ .
$$
\end{defin}


Observe that, if $T\neq 0$ and $\delta\in \mathbb{Z}$, the weight must be an even number: if $T$ is an homogeneous natural tensor of odd weight $\delta$, then the homogeneity condition for $\lambda=-1$ says:
$$T(\omega, \nabla)=T((-1)^2 \omega, \nabla)=(-1)^\delta T(\omega, \nabla)=-T(\omega, \nabla) \ ,$$
obtaining that $T=0$.


\medskip
\noindent {\bf Theorem \ref{MainThmFed}.} {\it
Let $X$ be a smooth manifold of dimension $2n$, and let $\F$ denote the sheaf of Fedosov structures. Let $\T$ be the sheaf of $p$-covariant tensors over $X$. Let $\delta \in \mathbb{Z}$.


Fixing a point $x_0 \in X$ and a chart $U\simeq \R2n$ around $x_0$ produces a $\mathbb{R}$-linear isomorphism

$$
\begin{CD}
\left\{
\begin{array}{c}
 \text{Natural morphisms of sheaves} \ \\
 \F \longrightarrow \T \ \\
 \text{homogeneous of weight }\delta \
\end{array} \right\} @=
\bigoplus \limits_{d_1, \ldots , d_r} \mathrm{Hom}_{\Sp}(S^{d_1}N_1 \otimes \ldots \otimes S^{d_r}N_r , T_{x_0} ) \ ,
\end{CD}
$$
where $\Sp=\Spp$ denotes the symplectic group, $T_{x_0}$ denotes the vector space of $p$-covariant tensors at $x_0$ and $d_1, \ldots , d_r$ run over the non-negative integer solutions of the equation }
\begin{equation}\label{eqMainFed}
2d_1 + \ldots + (r+1)d_r =p-\delta \ .
\end{equation}

\medskip



%

\begin{proof} 
Let us fix a point $x_0\in X$. Choose a chart $U\simeq \R2n$ around $x_0$, so that Proposition \ref{red1} produces a bijection:
$$
\begin{CD}
\left\{
\begin{array}{c}
   \text{Natural morphisms of sheaves}  \ \\
   \F \longrightarrow \T \ \\
\text{homogeneous of weight }\delta \
\end{array} \right\} @=
\left\{
\begin{array}{c}
    \text{Natural morphisms of sheaves}  \ \\
   \F_{\R2n} \longrightarrow \T_{\R2n} \ \\
 \text{homogeneous of weight }\delta \
\end{array} \right\} ,
\end{CD}
$$ 
where $\F_{\R2n}$ and $\T_{\R2n}$ denote the sheaves $\F$ and $\T$ restricted to $U$ and passed through the diffeomorphism $U\simeq \R2n$. 

Fixing the canonical symplectic form $\eta$ on $\R2n$ lets us invoke Proposition \ref{red2} and Proposition \ref{red3}, which gives the bijection:
$$
\begin{CD}
\left\{
\begin{array}{c}
    \text{Natural morphisms of sheaves}  \ \\
   \F_{\R2n} \longrightarrow \T_{\R2n} \ \\
 \text{homogeneous of weight }\delta \
\end{array} \right\} @=
\left\{
\begin{array}{c}
   \Autw_{x_0} \text{-equivariant smooth maps}  \ \\
  J_{x_0}^\infty \Conw  \longrightarrow T_{x_0} \ \\
 \text{homogeneous of weight }\delta \
\end{array} \right\} ,
\end{CD}
$$ 

where an $\FDiff$-equivariant smooth map $T:J_{x_0}^\infty\Conwpinf  \rightarrow T_{x_0}$ being homogeneous of weight $\delta$ means that it verifies the following property:
$$
T( h_\lambda  \cdot (j_{x_0}^\infty \nabla))=\lambda^{p-\delta}T(j_{x_0}^\infty \nabla) \ ,
$$
for any homothety\footnote{We say that $\tau \in \Diff$ is a homothety of ratio $\lambda \neq 0$ if $ \mathrm{d}_{x_0}\tau=\lambda \cdot \mathrm{Id}$.} $h_\lambda$ of ratio $\lambda\neq 0$. 

Let us now unfix the symplectic form (recall that diffeomorphisms act transitively on symplectic forms due to the existence of Darboux coordinates):

$$
\begin{CD}
\left\{
\begin{array}{c}
   \Autw_{x_0} \text{-equivariant smooth maps}  \ \\
  J_{x_0}^\infty \Conw  \longrightarrow T_{x_0} \ \\
 \text{homogeneous of weight }\delta \
\end{array} \right\} @=
\left\{
\begin{array}{c}
   \Diff \text{-equivariant smooth maps}  \ \\
  J_{x_0}^\infty \F \longrightarrow T_{x_0} \ \\
 \text{homogeneous of weight }\delta \
\end{array} \right\} .
\end{CD}
$$ 

As the action of both $\Diff$ and $\diffinfty$ coincide over $J^\infty_{x_0} \F$ and $T_{x_0}$, we may consider $\diffinfty$-equivariant maps instead in the set above. 

For the next step, recall that the following sequence of groups is exact:
\[
1 \longrightarrow \ndiffinfty \longrightarrow \diffinfty \longrightarrow \Gl \longrightarrow 1
\]

As the subgroup $\ndiffinfty$ acts by the identity over $T_{x_0}$, Corollary \ref{CorolarioCociente} in conjunction with the exact sequence above assures the existence of an isomorphism:

$$
\begin{CD}
\left\{
\begin{array}{c}
 \diffinfty \text{-equivariant smooth maps}  \ \\
  J_{x_0}^\infty \F  \longrightarrow T_{x_0} \ \\
 \text{homogeneous of weight }\delta \ 
\end{array} \right\} @=
\left\{
\begin{array}{c}
 \Gl \text{-equivariant smooth maps}  \ \\
  J_{x_0}^\infty \F / \ndiffinfty  \longrightarrow T_{x_0} \ \\
\text{homogeneous of weight }\delta \ 
\end{array} \right\}  
\end{CD} \ .
$$ 
\medskip

Now, Corollary \ref{CoroRedThm} allows us to replace this quotient ringed space via the bijection:

$$
\begin{CD}
\left\{
\begin{array}{c}
 \Gl \text{-equivariant smooth maps}  \ \\
  J_{x_0}^\infty \F / \ndiffinfty  \longrightarrow T_{x_0} \ \\
 \text{homogeneous of weight }\delta \ 
\end{array} \right\} @=
\left\{
\begin{array}{c}
  \Gl \text{-equivariant smooth maps}  \ \\
  \Lambda_0 \times \prod\limits_{i=1}^\infty N_i  \longrightarrow T_{x_0} \ \\
 \text{homogeneous of weight }\delta \ 
\end{array} \right\} \ .
\end{CD}
$$ 

Fixing the non-singular 2-form $\eta_{x_0}$ at $x_0$ allows us to remove the space $\Lambda_0$, due to the bijection:

$$
\begin{CD}
\left\{
\begin{array}{c}
  \Gl \text{-equivariant smooth maps}  \ \\
  \Lambda_0 \times \prod\limits_{i=1}^\infty N_i  \longrightarrow T_{x_0} \ \\
 \text{homogeneous of weight }\delta \ 
\end{array} \right\} @=
\left\{
\begin{array}{c}
  \Spp \text{-equivariant smooth maps}  \ \\
  \prod\limits_{i=1}^\infty N_i  \longrightarrow T_{x_0} \ \\
 \text{homogeneous of weight }\delta \ 
\end{array} \right\} \ ,
\end{CD}
$$ 
where, following the previous bijections, a $\Spp$-equivariant smooth map $T:\prod\limits_{i=1}^\infty N_i  \rightarrow T_{x_0}$ is said to be homogeneous of weight $\delta$ if, for any $\lambda \neq 0$, it holds that 

$$T(\lambda^2 T_1, \lambda^3 T_2, \ldots)=\lambda^{p-\delta} T(T_1,T_2, \ldots) \ .$$

Therefore, the homogeneity allows us to make the final reduction by applying the Homogeneous Function Theorem below, producing the isomorphism:

$$
\begin{CD}
\left\{
\begin{array}{c}
  \Spp \text{-equivariant smooth maps}  \ \\
   \prod\limits_{i=1}^\infty N_i
     \longrightarrow T_{x_0} \ \\
 \text{homogeneous of weight }\delta \ 
\end{array} \right\} @=
\bigoplus \limits_{d_1, \ldots , d_r} \mathrm{Hom}_{\Spp}(S^{d_1}N_1 \otimes \ldots \otimes S^{d_r}N_r , T_{x_0} ) \ ,
\end{CD}
$$
where $d_1, \ldots , d_r$ are non-negative integers running over the solutions of the equation
\[
2d_1 + \ldots + (r+1)d_r =p-\delta \ .
\]

\qed

\end{proof}

\begin{funchom} \label{FuncionesHomogeneas} {\it Let $\{ E_i \}_{i\in \mathbb{N}}$ be finite dimensional vector spaces.


Let $\, f \, \colon  \prod_{i=1}^\infty E_i \to \mathbb{R}$ be a smooth function such that there exist positive real numbers $a_i > 0$, and $ w \in \mathbb{R}$ satisfying:
\begin{equation}\label{CondicionHomogeneidadLemma}
 f ( \lambda^{a_1} e_1 , \ldots , \lambda^{a_i} e_i , \ldots ) = \lambda^w \, f(e_1 , \ldots , e_i , \ldots )
\end{equation}
 for any positive real number $\lambda >0$ and any
$(e_1 , \ldots , e_i , \ldots ) \in \prod_{i=1}^\infty E_i$.

Then, $f$ depends on a finite number of variables $e_1,\dots, e_r$ and it is a sum of monomials of 
degree $d_i$ in $e_i$ satisfying the relation
\begin{equation}\label{CondicionMonomios}
 a_1 d_1 + \cdots + a_r d_r = w \ .
\end{equation}
 
If there are no natural numbers  $d_1,\dots,d_r \in \mathbb{N} \cup \{ 0 \}$ satisfying this equation, then $f$ is the zero map. }
\end{funchom}

An immediate corolary of the Main theorem is that, if the left side of Equation \ref{eqMainFed} is either null or negative, there are essentially no natural tensors:

\begin{cor}
There are no non-constant homogeneous natural $p$-tensors associated to Fedosov structures of weight $\delta\geq p$.
\end{cor}

\bigskip

\subsection{An application}

Let $\,V\,$ be a real vector space of finite dimension $\,2n$, let $\omega$ be a non-degenerate skew-symmetric bilinear form on $V$ and let $\,\Spp \,$ be the real Lie group of  $\mathbb{R}$-linear automorphisms that preserve $\omega$. 

The First Fundamental Theorem of the symplectic group (\cite{GW}) describes the vector space of $\Spp$-invariant linear maps $ V \otimes \stackrel{p}{\ldots} \otimes V \, \longrightarrow \, \RR \ : $

\begin{funthm}\label{MainTheoremSp} 
The real vector space $\,\mathrm{Hom}_{\Spp}\left( V \otimes \stackrel{p}{\ldots} \otimes V  \, , \, \RR \right) \,$
of invariant linear forms on $\, V \otimes \ldots \otimes V\,$ is null if $p$ is odd, whereas if $p$ is even it is spanned by 
$$ \omega_\sigma ((e_1 , \ldots , e_p)) := \omega (e_{\sigma(1)}, e_{\sigma(2)}) \ldots \omega (e_{\sigma(p-1)}, e_{\sigma(p)}) \ ,   $$
where $\sigma \in S_p .$
\end{funthm}

The invariant theory of the symplectic group, along with our Main Theorem, allows us to compute the space of natural functions for weights $w=-2$ and $w=-4$:

\begin{cor}
There are no non-constant homogeneous natural functions associated to Fedosov structures of weight $w=-2$, and for $w=-4$ there are three $\RR$-linearly independent natural functions.
\end{cor}

\begin{proof}
Let us fix $x_0\in X$ and a non-singular 2-form $\omega$ at $x_0$. Let us invoke the Main Theorem \ref{MainThmFed} for $p=0$ and $\delta=-2$. The only non-negative integer solution of the equation
$$2d_1 + \ldots + (r+1)d_r =p-\delta=2$$
is $d_1=1$. 

Therefore, the problem is reduced to computing $\Spp$-equivariant maps $N_1 \rightarrow \RR$. As the elements in $N_1$ are $4$-covariant tensors symmetric in the second and third indices, by the First Fundamental theorem of $\Sp$ it is sufficient to check that the map
$$T_{ijka} \longrightarrow \omega^{ij} \omega^{ka} T_{ijka}$$
is zero:
$$\omega^{ij} \omega^{ka} T_{ijka}=\frac{1}{2} \omega^{ij} \omega^{ka}(T_{ijka}-T_{jika})=0 \ ,$$
as the elements in $N_1$ verify that
$$T_{ijka}-T_{jika}=T_{ijak}-T_{jiak} \ .$$

Repeating the arguments for $p=0$ and $w=-4$, we obtain two solutions to the equation above: $d_1=2$ and $d_3=1$. Let us begin with solution $d_1=2$: we need to compute total index contractions of the expression $T_{ijkl}T_{abcd}$. Equivalently, we may replace this expression by applying the $\Sp$-equivariant linear isomorphism 
\begin{align*}
N_1 &\longrightarrow \mathcal{R} \\
T_{ijkl} &\longmapsto R_{ijkl}=T_{ijlk}-T_{ijkl},
\end{align*}
where $\mathcal{R}\subset S^2T_{x_0}^*X \otimes \Lambda^2 T_{x_0}^*X$ is the vector subspace of tensors $R$ that satisfy the Bianchi identity:
$$R_{ijkl}+R_{iklj}+R_{iljk}=0 \ .$$

Thus, let us compute the total index contractions of the expression $R_{ijkl}R_{abcd}$. As the contraction of the symmetric pair is zero, the possibilities are:
\begin{itemize}
\item $f_1=R_{ijkl}R^{ijkl}$.
\item $f_2=R_{ijk}^{\ \ \ k} R^{ijl}_{\ \ \ l}$.
\item $R_{ijkl}R^{ikjl}$, which is equal to $f_1/2$, by the Bianchi identity.
\end{itemize}
For $d_3=1$, the last three indices of any tensor in $N_3$ are symmetric, so there is only one possibility: $f_3=T_{ijk}^{\ \ \ ijk}$.

As for the linear independence of the three functions, by naturalness it is enough to check if they are independent at any given Fedosov manifold. For example, consider the Fedosov manifold $(\RR^4, \eta, \nabla)$, where $\eta=\d x_1 \wedge \d x_2 + \d x_3 \wedge \d x_4$ and $\nabla$ is the linear connection with the following Christoffel symbols (with the contravariant index lowered):
\begin{itemize}
\item $\Gamma_{ijk}=1$, for any $\{i,j,k\}$ permutation of $\{1,1,2\}$.
\item $\Gamma_{ijk}=x_1x_3x_4$, for any $\{i,j,k\}$ permutation of $\{2,3,4\}$.
\item $\Gamma_{ijk}=0$, for any other combination.
\end{itemize} 

Computing the natural functions in this manifold gives:

\begin{itemize}
\item $f_1=-4x_3^2x_4^2(- 4x_1^2 + 4x_1 + 1)$.
\item $f_2=2x_3^2x_4^2(4x_1^2 - 1)$.
\item $f_3=6$,
\end{itemize}
which are clearly $\RR$-linearly independent.
\qed
\end{proof}

\bigskip

\noindent \textbf{Acknowledgements.}  The authors would like to thank Professor Juan B.  Sancho de Salas for his generous advice.

\end{document}